\documentclass[12pt]{article}
\usepackage{macros}

\title{Whitehead torsion and the kernel of assembly}
\author{Oscar Harr}
\date{\today}

\begin{document}
\maketitle
\begin{abstract}
  For a topological space that is homeomorphic to a finite simplicial complex, we prove that the Bartels--Nikolaus assembly functor has a fully faithful right adjoint. Using this, we define for each such topological space $X$ a {\em Whitehead category}, whose K-theory is canonically identified with the Whitehead spectrum of $X$; and for a homotopy equivalence between two such spaces, we define an object in the Whitehead category of $X$ called the {\em torsion cosheaf} of the map, whose K-theory class recovers the classical Whitehead torsion.
\end{abstract}
\begin{small}
  \tableofcontents
\end{small}
\setcounter{section}{-1}
\newpage
\section{Introduction}
Let $X$ and $Y$ be triangulated spaces. A map $f\colon Y\to X$ is said to be a \tdef{simple homotopy equivalence} if up to homotopy it factors as a sequence of special homotopy equivalences coming from the combinatorics of simplicial complexes, called \tdef{elementary collapses} and \tdef{elementary expansions} (see~\cref{fig:elem-collapse}). It is natural to ask to what extent the homotopy theory of triangulated spaces is captured by the combinatorics of such collapses and expansions, or more precisely whether a homotopy equivalence $f\colon Y\to X$ is automatically a simple homotopy equivalence. A celebrated theorem of Whitehead gives a complete answer to this question, in the form of an obstruction class
\[
  \uptau (f)\in\WhiteheadSpt{X}[1]
\]
that vanishes if and only if $f$ is a simple homotopy equivalence \cite{Whitehead1950}. Here $\WhiteheadSpt{X}[1]$ is the so-called \tdef{Whitehead group} of $X$, which only depends on the fundamental groupoid of $X$, and $\uptau (f)$ is called the \tdef{Whitehead torsion} of $f$.

The purpose of this article is to promote the Whitehead group to a category $\WhiteheadCat X$ and the Whitehead torsion to an object in this category $\Tors f\in\WhiteheadCat X$. Concretely, we will have a canonical identification of $\WhiteheadSpt{X}[1]$ with the K-theory $\K_0(\WhiteheadCat X)$ under which $\uptau (f)$ is equal to the class $\lbrack\Tors f\rbrack$. For this to make sense, $\WhiteheadCat X$ had better belong to a class of categories for which K-theory is a well-behaved notion. For us this class will be the dualizable stable $\infty$-categories and K-theory will be understood in the sense of Efimov's continuous K-theory~\cite{Efimov2025}.
\begin{mainthm}\label{intro whitehead category}
  Let $\canr$ denote the category whose objects are compact absolute neighborhood retracts (ANRs) and whose morphisms are continuous maps. Let $\Prdual$ denote the $\infty$-category of dualizable stable $\infty$-categories and strongly continuous functors. There is a functor
  \[
    \WhiteheadCat{-}\colon\canr\to\Prdual
  \]
  and a natural transformation of functors valued in the $\infty$-category of spectra
  \begin{equation}
    \label{eq:transform-from-k-to-whitehead}
    \K\circ\WhiteheadCat{-}\to\desusp\WhiteheadSpt{-},
  \end{equation}
  where $X\mapsto\WhiteheadSpt{X}$ is the functor that sends $X\in\canr$ to its Whitehead spectrum. Furthermore,
  \begin{enumerate}[label=(\roman*)]
  \item The map $\K (\WhiteheadCat{X})\to\desusp\WhiteheadSpt X$ is an equivalence if $X$ admits a triangulation. In particular, for such $X$ we have $\K_0(\WhiteheadCat X)\simeq\WhiteheadSpt{X}[1]$.
  \item Given a morphism $f\colon Y\to X$ in $\canr$, there is a compact object $\Tors f\in\WhiteheadCat X$ such that if $X$ admits a triangulation then $\lbrack\Tors f\rbrack\in\K_0(\WhiteheadCat X)$ is equal to $\uptau (f)$ under the identification from (ii).
  \end{enumerate}
\end{mainthm}
\begin{figure}[h]
  \centering
  \def\svgwidth{0.8\textwidth}
\begingroup%
  \makeatletter%
  \providecommand\color[2][]{%
    \errmessage{(Inkscape) Color is used for the text in Inkscape, but the package 'color.sty' is not loaded}%
    \renewcommand\color[2][]{}%
  }%
  \providecommand\transparent[1]{%
    \errmessage{(Inkscape) Transparency is used (non-zero) for the text in Inkscape, but the package 'transparent.sty' is not loaded}%
    \renewcommand\transparent[1]{}%
  }%
  \providecommand\rotatebox[2]{#2}%
  \newcommand*\fsize{\dimexpr\f@size pt\relax}%
  \newcommand*\lineheight[1]{\fontsize{\fsize}{#1\fsize}\selectfont}%
  \ifx\svgwidth\undefined%
    \setlength{\unitlength}{283.46456693bp}%
    \ifx\svgscale\undefined%
      \relax%
    \else%
      \setlength{\unitlength}{\unitlength * \real{\svgscale}}%
    \fi%
  \else%
    \setlength{\unitlength}{\svgwidth}%
  \fi%
  \global\let\svgwidth\undefined%
  \global\let\svgscale\undefined%
  \makeatother%
  \begin{picture}(1,0.22238001)%
    \lineheight{1}%
    \setlength\tabcolsep{0pt}%
    \put(0,0){\includegraphics[width=\unitlength,page=1]{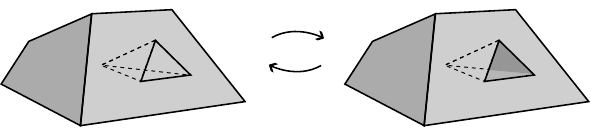}}%
    \put(0.44939416,0.18947319){\color[rgb]{0,0,0}\makebox(0,0)[lt]{\lineheight{1.25}\smash{\begin{tabular}[t]{l}\footnotesize{collapse}\end{tabular}}}}%
    \put(0.26099236,0.10435791){\color[rgb]{0,0,0}\makebox(0,0)[lt]{\lineheight{1.25}\smash{\begin{tabular}[t]{l}$\sigma$\end{tabular}}}}%
    \put(0.2983355,0.00108169){\color[rgb]{0,0,0}\makebox(0,0)[lt]{\lineheight{1.25}\smash{\begin{tabular}[t]{l}$\Sigma$\end{tabular}}}}%
    \put(0.89100236,0.00108169){\color[rgb]{0,0,0}\makebox(0,0)[lt]{\lineheight{1.25}\smash{\begin{tabular}[t]{l}$\Sigma '$\end{tabular}}}}%
    \put(0.44236802,0.06945129){\color[rgb]{0,0,0}\makebox(0,0)[lt]{\lineheight{1.25}\smash{\begin{tabular}[t]{l}\footnotesize{expansion}\end{tabular}}}}%
  \end{picture}%
\endgroup%

  \caption{An elementary collapse of a complex $\Sigma$ is a subcomplex $\Sigma '\subset\Sigma$ such that $\Sigma = \Sigma '\cup_{C (\partial\sigma)}C\sigma$ for some simplex $\sigma$. On geometric realizations, there is an associated retract $|\Sigma |\to |\Sigma '|$ (unique up to contractible choice) supported on the cone $C\sigma$. Conversely, the inclusion $|\Sigma '|\hookrightarrow |\Sigma |$ is called an elementary expansion. The figure shows an elementary collapse in which $\sigma$ is a 2-simplex.}
  \label{fig:elem-collapse}
\end{figure}
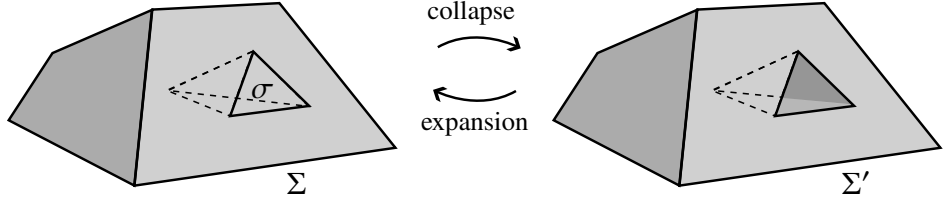

By the Cairns--Whitehead triangulation theorem \cite{Cairns1935,Whitehead1940}, the conclusions in (i) and (ii) apply in particular to compact $C^1$-manifolds.

\medskip

The most difficult part of~\cref{intro whitehead category} is (i), which we deduce from the following technical theorem:
\begin{mainthm}[\cref{thm: epimorphism}]\label{intro epimorphism}
  If $X$ is a topological space that has the homeomorphism type of a finite simplicial complex, then the Bartels--Nikolaus assembly functor (see~\cref{subsec: bartels nikolaus})
  \begin{equation}
    \label{eq:intro-bnass}
    \BNass\colon\ccoShv {X}\to\Sp^{\type X}
  \end{equation}
  has a fully faithful right adjoint.
\end{mainthm}
Another proof of this result has been found independently by Nikolaus and Ramzi. Our proof is based on codescent to the case where $X$ is contractible. For this to work, we need to show that the $\infty$-category appearing on the left-hand side of~\eqref{eq:intro-bnass} has codescent:
\begin{mainthm}[\cref{thm: codescent}]\label{intro codescent}
  Let $X$ be a second-countable and locally compact ANR. If $Y_0$ and $Y_1$ are closed subspaces of $X$ with $Y_{01} = Y_0\cap Y_1$ such that $Y_0$, $Y_1$, and $Y_{01}$ are again ANRs, then the square
  \begin{equation*}
    \begin{tikzcd}
      \ccoShv{Y_{01}}\arrow[r]\arrow[d]
      & \ccoShv{Y_{1}} \arrow[d] \\
      \ccoShv{Y_{0}}\arrow[r]
      & \ccoShv{Y}
    \end{tikzcd}
  \end{equation*}
  is a pushout in the $\infty$-category of dualizable categories.
\end{mainthm}
\subsection{Sales pitch}
There are two advantages to working with Whitehead torsion via the machinery provided by $\WhiteheadCat X$ and $\Tors f$:

Firstly, the $\infty$-category $\WhiteheadCat X$ is a much richer object than the Whitehead spectrum $\WhiteheadSpt X$; whereas the only structure possessed by the latter is addition and its derived consequences, the former has limits, colimits, filtrations, mapping spectra, a notion of support, etc. Furthermore, one can define useful functors $\WhiteheadCat X\to\WhiteheadCat Y$ besides the ones that induce the usual maps on Whitehead spectra.

Secondly,
\begin{quote}
  \emph{none of the categories, objects, functors, natural transformations, or identifications mentioned in~\cref{intro whitehead category} depend on a choice of triangulation(s)}.
\end{quote}
In particular, note that the conclusions in part (i) and (ii) of the theorem only depend on $X$ having the {\em property} of admitting a triangulation. In fact, we prove (i) and (ii) for a wider class of spaces, namely those that have a ``good cover\footnote{in the sense of the Borsuk nerve theorem} by closed subspaces''. A common technical burden affecting geometric topology is the tension between discrete defined invariants (e.g.~Wall class, Whitehead torsion) and continuous techniques (e.g.~flows). In our approach to Whitehead torsion, the object $\Tors f$ is directly amenable to continuous methods.

As a bonus point, the object $\Tors f$ has a very concrete description: it is the cosheaf that measures the extent to which $f\colon X\to Y $ is a homotopy equivalence {\em locally} on $X$. This lends itself well to the ideas of \tdef{controlled topology} in the sense of Chapman, Ferry, and Quinn. We will explore this connection further in ongoing work with Christian Kremer

\subsection{What we do in this paper}
In~\cref{sec: whitehead cat}, we first give an overview of the dependencies that we import from Bartels and Nikolaus's work on simple homotopy theory and geometric topology. After this, we define the functor and natural transformation appearing in~\cref{intro whitehead category}. We also prove~\cref{intro epimorphism} and~\cref{intro codescent}, from which we deduce part (i) of~\cref{intro whitehead category}. In~\cref{sec: htpy eq}, we define the object $\Tors f$ for a map between compact ANRs, and show part (ii) of~\cref{intro whitehead category} along with some basic properties of $\Tors f$.  
\subsection{Acknowledgements}
I am deeply grateful to Lars Hesselholt for posing the questions on which this article is based. I also wish to thank Maxime Ramzi for illuminating discussions, especially about the construction of $\Tors f$, and Christian Kremer for asking about the relationship between $\Tors f$ and controlled topology. Finally, thanks to Marceline Jin-Harr for many insightful comments.
I was partially supported by the Danish National Research Foundation through the Copenhagen Centre for Geometry and Topology (DNRF151), as well as Dan Petersen's Wallenberg Scholar fellowship.

\section{Conventions}
We use the theory of $\infty$-categories, generally following Lurie~\cite{Lurie2009,Lurie2017} in terminology. In fact, as is becoming common we will simply refer to an $\infty$-category as a category, and write 1-category if we wish to stress that a category is a category in the classical sense. 
\begin{itemize}
\item We refer to objects in the $\infty$-category of ``spaces''/$\infty$-groupoids/homotopy types as \tdef{anima}.
\item For convenience, we fix uncountable Grothendieck universes $\mathcal U\in\mathcal U'\in\mathcal U''$ of \tdef{small}, \tdef{large}, and \tdef{very large} sets. Topological spaces, anima, and spectra are assumed to be small; categories are (unless otherwise stated) large. 
\item Given a topological space $X$, we let $\mdef{\type X}$ denote its associated anima.
\item We use the formalism of six operations for spectral sheaves on locally compact Hausdorff spaces (see Volpe's article~\cite{Volpe2025}), which associates to a locally compact Hausdorff space its category of spectral sheaves $\Shv X$ and to a continuous map $f\colon Y\to X$ a pair of adjunctions $f^*\dashv f_*$ (pullback/pushforward) and $f_!\dashv f^!$ (exceptional pushforward/pullback). We follow the notation used by Volpe.
\item We also use the formalism of five operations on parametrized spectra on anima. The classical reference for this theory is~\cite{May_Sigurdsson2006}, which is written in the language of model categories. One can find some useful references for the $\infty$-categorical enhancement of this theory scattered throughout the literature, see~\cites[][App.~A]{Land2022}[][]{Cnossen2023}[][Sec.~4]{Picot2025}. This theory associates to an anima $K$ its category of parametrized spectra $\Sp^K = \Fun (K,\Sp )$ and to a map $f\colon K\to L$ of anima a triple of adjunctions $f_\sharp\dashv f^*\dashv f_*$.
\item Given a topological space $X$, we abuse notation by writing $X\colon X\to\pt$ also for the projection from $X$ to a point. We follow the same convention for anima. Note that this is consistent with the way adjectives work in these categories (and generally in the presence of six-functor formalisms); a topological space $X$ is proper (i.e.~compact) if and only if the projection from $X$ to a point is proper.
\item Given a topological space $X$ that is locally of singular shape in the sense of Lurie, we let $\mdef{\psi_X^*}\colon\Sp^{\type X}\hookrightarrow\Shv X$ denote the inclusion of local systems into sheaves. This functor has a left adjoint $(\psi_X)_\sharp$ and a right adjoint $(\psi_X)_*$.
\item Given presentable stable categories $\mathcal C$ and $\mathcal D$, we let $\mdef{\Fun^L(\mathcal C,\mathcal D)}$ denote the category of colimit-preserving functors from $\mathcal C$ to $\mathcal D$, or equivalently the category of functors that admit a right adjoint. We let $\mdef{\Fun^{LL}(\mathcal C,\mathcal D)}\subseteq \Fun^L(\mathcal C,\mathcal D)$ denote the full subcategory spanned by \tdef{strongly continuous} functors; i.e. colimit-preserving functors whose right adjoint admits a further right adjoint.
\item We let $\mdef{\largecats}$ denote the (very large) category of large categories. We also let $\mdef{\Prlst}$ (resp. $\mdef{\Prdual}$) denote the (very large) category of presentable stable $\infty$-categories (resp. dualizable stable $\infty$-categories) and colimit-preserving (resp. strongly continuous) functors between them.
\end{itemize}

\section{The Whitehead category}
\label{sec: whitehead cat}
\subsection{The Bartels--Nikolaus assembly functor}\label{subsec: bartels nikolaus}
Our work is based on Bartels and Nikolaus's work on geometry topology and simple homotopy theory, which will appear in an upcoming article of theirs. In this subsection, we give an overview of the definitions and results that we import from this theory. For a broader account, we refer to Nikolaus's excellent lecture series on this topic~\cite{Nikolaus2024}, as well as the notes~\cite{Krause_Nikolaus_Putzstuck2024} for an expository treatment of the theory of dualizable categories and continuous K-theory.

\medskip

The Lurie tensor product on presentable categories restricts to a symmetric monoidal structure on the (very large) category of dualizable stable categories $\Prdual$. This symmetric monoidal structure is closed, and we let $\Homd (-,-)\colon (\Prdual )^{\mathrm{op}}\times\Prdual\to\Prdual$ denote the associated internal mapping object. This object is generally very different from the internal mapping object $\Fun^L(-,-)$ in $\Prlst$.
\begin{defn}
  The category of \tdef{completed cosheaves} on a locally compact Hausdorff space $X$, denoted $\mdef{\ccoShv X}$, is the dualizable stable category
  \[
    \ccoShv X = \Homd (\Shv X,\Sp ).
  \]
\end{defn}
That is, $\ccoShv X$ is the functional/weak dual of $\Shv X$ in the category of dualizable stable categories. It is essentially never a strong dual to $\Shv X$, as the latter is dualizable in $\Prdual$ if and only if $X$ is finite and discrete \cite{Harr2023b}.

\begin{rmk}
  The assignment $X\mapsto\Shv X\in\Prdual$ is contravariantly functorial in proper maps, since we can send $f\colon Y\to X$ to the pullback $f^*\colon\Shv X\to\Shv Y$, which is strongly continuous on account of the adjunctions $f^*\dashv f_*\simeq f_!\dashv f^!$. It follows that $X\mapsto\ccoShv X$ is {\em covariantly} functorial in proper maps. The functor induced by a proper map $f\colon Y\to X$ is denoted $\mdef{f_?}\colon\ccoShv Y\to\ccoShv X$.
\end{rmk}
\begin{rmk}
  General objects of $\ccoShv X$ are somewhat unwieldy\footnote{There is however a convenient collection of objects that generate $\ccoShv X$ under colimits \cite[Prop~3.12]{Efimov2025b}, as we will use later}. On the other hand, compact objects of $\ccoShv X$ are classified by strongly continuous functors $\Sp\to\ccoShv X$, so by the defining property of the internal mapping object we have
  \[
    \ccoShv X^{\omega}\simeq\Fun^{LL} (\Sp ,\ccoShv X)\simeq\Fun^{LL}(\Shv X,\Sp ).
  \]
  That is, $\ccoShv X^{\omega}$ identifies with the full subcategory $\Fun^{LL}(\Shv X,\Sp )\subseteq\Fun^L(\Shv X,\Sp ) = \coShv X$ spanned by strongly continuous functors.
\end{rmk}

The centerpiece Bartels and Nikolaus's work is their assembly functor, whose definition we now recall. Let $X$ be a locally compact ANR with underlying anima $K=\type X$, and consider the composite functor
\begin{align}
  \begin{split}
    \label{eq:deltaR}
    &\Shv X\otimes\Sp^K\xrightarrow{\id\otimes\psi_X^*}\Shv X\otimes\Shv X\simeq\Shv{X\times X}  \\
    &\qquad\qquad\xrightarrow{\Delta^*}\Shv X\xrightarrow{X_*}\Sp.
  \end{split}
\end{align}
\begin{prop}[Bartels--Nikolaus]
  If $X$ is a compact ANR, then~\eqref{eq:deltaR} admits a left adjoint.
\end{prop}
We let $\mdef{\delta}\colon\Sp\to\Shv X\otimes\Sp^K$ denote the left adjoint of~\eqref{eq:deltaR}.
\begin{defn}[Bartels--Nikolaus]
  Let $X$ be a compact ANR with underlying anima $K=\type X$, and let $\delta$ be defined as above. The \tdef{Bartels--Nikolaus assembly functor}, denoted  $\mdef{\BNass}$, is the composite
  \begin{align}
  \begin{split}
    \label{eq:ass-defn}
    &\ccoShv X\simeq\ccoShv X\otimes\Sp\xrightarrow{\id\otimes\delta}\ccoShv X\otimes\Shv X\otimes\Sp^K\\
    &\qquad\qquad\xrightarrow{\mathrm{ev}\otimes\id}\Sp^K.
  \end{split}
\end{align}
\end{defn}
It is straightforward to verify
\begin{prop}
  The Bartels--Nikolaus assembly functor is natural in proper maps $f\colon Y\to X$, where $f\colon Y\to X$ acts as $f_?$ on $\ccoShv Y$ and $f_\sharp$ on $\Sp^{\type Y}$.
\end{prop}
The Bartels--Nikolaus assembly functor is so-named because it categorifies Waldhausen's assembly map \cite{Waldhausen1978}. Recall that the latter is defined for a given anima $K$ as the canonical comparison map
\[
  \A(\pt )\otimes K\simeq \colim_K\A (\pt )\to \A (\colim_K\pt )\simeq\A (K),
\]
where $\mdef{\A (-)}$ is the functor that sends an anima $K$ to the K-theory of its spherical group ring, that is $\K (\SS\lbrack\desusp K\rbrack)$. We then have
\begin{thm}[Bartels--Efimov--Nikolaus]
  For $X$ a compact ANR, the map
  \[
    \K (\BNass )\colon\K (\ccoShv X)\to\K (\Sp^{\type X})
  \]
  is naturally identified with the assembly map for $\type X$.
\end{thm}
Note that $\K (\Sp^{\type X})\simeq \A (\type X)$ by definition. On the other hand, the most difficult part of the preceding theorem is the identification of the target $\K (\ccoShv X)$ with $\A (\pt )\otimes\type X$, which is a special case of~\cite[Thm~11.1]{Efimov2025c}. In fact, once one has this calculation, it is not too hard to deduce the theorem from the case where $X$ is a simplex via a codescent argument like the one we use to prove~\cref{intro epimorphism}.

\medskip

Aside from the assembly functor, we need one more construction from Bartels and Nikolaus's work:
\begin{defn}[Bartels--Nikolaus]
  Let $X$ be a locally compact ANR. The \tdef{local Wall object} of $X$, denoted $\mdef{\WallObject X}$, is the compact object of $\ccoShv X$ associated with the strongly continuous functor $X_\sharp\colon\Shv X\to\Sp$.
\end{defn}
That is, $\WallObject X$ is associated with the cosheaf $U\mapsto\SS\lbrack \type U\rbrack$.
\begin{thm}[Bartels--Nikolaus]
  For $X$ a compact ANR, the assembly functor $\BNass$ takes $\WallObject X$ to the constant sphere spectrum $\SS_{\type X}\in\Sp^{\type X}$.
\end{thm}
It is a consequence of this theorem that the K-theory class $\lbrack\WallObject X\rbrack\in\A (\pt )\otimes\type X$ captures the simple homotopy type of $X$, see~\cite{Nikolaus2024}. In particular, it follows that every compact ANR has a well-defined simple homotopy type, which recovers and extends a celebrated theorem of Chapman. We will not have need of this here, and mention it only to advertise the work of Bartels and Nikolaus.  
\subsection{The Whitehead category and the fiber of assembly}
The \tdef{Whitehead spectrum} of an anima $K$, denoted $\mdef{\WhiteheadSpt{K}}$, is defined to be the cofiber of its assembly map; that is, there is a fiber sequence
\begin{equation}
  \label{eq:waldhausen-fiber-seq}
    \A (\pt )\otimes K\to \A (K)\to \WhiteheadSpt K.
\end{equation}
It is well-known that $\pi_1(\WhiteheadSpt K)$ is naturally isomorphic to the Whitehead group $\WhiteheadSpt{K}[1]$.
Often it is more natural to consider the {\em fiber} of the assembly map, which is then the desuspension of the Whitehead spectrum $\desusp\WhiteheadSpt K$.
Building on Bartels and Nikolaus's work, we categorify $\desusp\WhiteheadSpt K$ by simply attempting to lift the fiber sequence~\eqref{eq:waldhausen-fiber-seq} (or rather its rotated form with fiber $\desusp \WhiteheadSpt{K}[1]$) from spectra to dualizable categories:
\begin{defn}
  If $X$ is a compact ANR, we define the \tdef{Whitehead category} of $X$, denoted $\mdef{\WhiteheadCat X}$, by the formula
  \[
    \WhiteheadCat X = \Ker (\BNass\colon\ccoShv {X}\to\Sp^{\type X}),
  \]
  where $\BNass$ is the Bartels--Nikolaus assembly functor.
\end{defn}
\begin{rmk}
  Since $\BNass$ is natural in continuous maps between compact ANRs, the assignment $X\mapsto\WhiteheadCat X$ is functorial in such maps. Furthermore, by construction we get a nullhomotopy of the composition
  \[
    \K (\WhiteheadCat X)\to \K (\ccoShv X)\to \K (\Sp^{\type X}),
  \]
  whence we get a natural transformation from $\K (\WhiteheadCat X)$ to the fiber of $\K (\ccoShv X)\to \K (\Sp^{\type X})$. The latter is canonically identified with the fiber of Waldhausen's assembly map, namely $\desusp{\WhiteheadSpt X}$. We thus get a canonical map
  \begin{equation}
    \label{eq:k-theory-vs-desuspension-of-whitehead-spectrum}
    \K (\WhiteheadCat X)\to \desusp \WhiteheadSpt{\type X},
  \end{equation}
  natural in $X$.
\end{rmk}
In this section, we will show that the map~\eqref{eq:k-theory-vs-desuspension-of-whitehead-spectrum} is an equivalence for a class of nice compact spaces, namely those admitting a decent cover in the sense of~\cref{subsec: decent covers}. By~\cref{prop:triangulation gives decent cover}, this class of spaces includes all spaces that admit a triangulation, thereby proving part (i) of~\cref{intro whitehead category}.

We can deduce the desired statement about~\eqref{eq:k-theory-vs-desuspension-of-whitehead-spectrum} from~\cref{intro epimorphism}, which we restate here for the reader's convenience:
\begin{thm}\label{thm: epimorphism}
  If $X$ is a compact ANR that admits a decent cover, then the Bartels--Nikolaus assembly functor $\BNass\colon\ccoShv {X}\to\Sp^{\type X}$ has a fully faithful right adjoint.
\end{thm}
Indeed, if the Bartels--Nikolaus assembly functor has a fully faithful right adjoint (or in other words, is a reflective localization), then
\[
  \WhiteheadCat X\to\ccoShv X\to\Sp^{\type X}
\]
is a Verdier sequence, and it follows from the fact that K-theory is a localizing invariant that
\begin{cor}
  If $X$ is a compact ANR such that its Bartels--Nikolaus assembly functor satisfies the conclusion of~\cref{thm: epimorphism}, then~\eqref{eq:k-theory-vs-desuspension-of-whitehead-spectrum} is an equivalence.
\end{cor}
\subsection{Decent covers}\label{subsec: decent covers}
Recall that an open cover $\lbrace U_i\rbrace_{i\in I}$ of a paracompact space $X$ is said to be a \tdef{good cover} if the intersection $\bigcap_{j\in J}U_j$ is either empty or contractible for every finite subset $J\subseteq I$. We will need a similar condition on closed covers:
\begin{defn}
  Let $X$ be a compact ANR, and let $\lbrace Y_i\rbrace_{i\in I}$ be a cover of $X$ by closed subsets. We will write $Y_J = \bigcap_{j\in J}Y_j$ for a subset $J\subseteq I$, so in particular $Y_\varnothing = X$. The cover $\lbrace Y_i\rbrace_{i\in I}$ is said to be a \tdef{decent cover} if the following conditions are satisfied:
  \begin{enumerate}[label=(\alph*)]
  \item The interiors $\lbrace Y_i^\circ\rbrace_{i\in I}$ form an open cover of $X$; and 
  \item For each finite subset $J\subseteq I$, the space $Y_J$ is itself an ANR and either contractible or empty.
  \end{enumerate}
\end{defn}
As with good covers, not every space admits a decent cover. The following proposition says that many of the spaces we care about do:
\begin{prop}\label{prop:triangulation gives decent cover}
  If $X$ is the geometric realization of a simplicial complex, then $X$ admits a decent cover.
\end{prop}
\begin{proof}
  Let $|\Sigma |\xrightarrow\sim X$ be a triangulation of $X$. The collection of closed stars $|\mathrm{Star}_\Sigma (\sigma )|$ forms a decent cover of $X$. Indeed, the closed star $|\mathrm{Star}_\Sigma (\sigma )|$ deformation retracts onto $|\sigma |$, which is contractible. It remains for us to show that intersections of closed stars satisfy the necessary conditions. For this, note that if $\sigma_1,\dots ,\sigma_n$ is a collection of simplices, then $\mathrm{Star}_\Sigma (\sigma_1 )\cap\cdots\cap \mathrm{Star}_\Sigma (\sigma_n )$ is empty unless $\sigma_1\cup\cdots\cup\sigma_n$ is a simplex in $\Sigma$; if the latter holds, then $\mathrm{Star}_\Sigma (\sigma_1 )\cap\cdots\cap \mathrm{Star}_\Sigma (\sigma_n ) = \mathrm{Star}_{\Sigma }(\sigma_1\cup\cdots\cup\sigma_n )$, and so we are in the case already treated.
\end{proof}
In particular, by the Cairns--Whitehead triangulation theorem \cite{Cairns1935,Whitehead1940}, we conclude that every paracompact $C^1$-manifold admits a decent cover.
\subsection{Codescent for completed cosheaves}
The key technical ingredient in our proof is the fact that $\ccoShv{-}$ satisfies {\em co}descent with respect to finite closed covers. It is well-known that $\Shv{-}$ satisfies descent with respect to finite closed covers, but the statement about $\ccoShv{-}$ is not a formal consequence of this. (For formal reasons, we have that $\Homd ({-},{-})$ takes {\em colimits} in the contravariant variable to limits; we are asking for a {\em limit} to be taken to a colimit.) Instead, codescent for $\ccoShv{-}$ follows from some difficult results of Efimov~\cite{Efimov2025b}.
\begin{thm}\label{thm: codescent}
  Let $X$ be a second-countable and locally compact ANR. If $Y_0$ and $Y_1$ are closed subspaces of $X$ with $Y_{01} = Y_0\cap Y_1$ such that $Y_0$, $Y_1$, and $Y_{01}$ are again ANRs, then the square
  \begin{equation}
    \label{eq:codescent-diagram}
    \begin{tikzcd}
      \ccoShv{Y_{01}}\arrow[r]\arrow[d]
      & \ccoShv{Y_{1}} \arrow[d] \\
      \ccoShv{Y_{0}}\arrow[r]
      & \ccoShv{Y}
    \end{tikzcd}
  \end{equation}
  is a pushout in the category of dualizable categories.
\end{thm}
We will need a few lemmata. The first is often used without comment in the literature:
\begin{lem}\label{lem:beck-chevalley}
  Let
  \begin{equation}
    \label{eq:fully-faithfullness-right-adjointability}
    \begin{tikzcd}
      \mathcal C\arrow[r,"f"]\arrow[d,"i"]
      & \mathcal C'\arrow[d,"i'"] \\
      \mathcal D\arrow[r,"g"]
      & \mathcal D'
    \end{tikzcd}
  \end{equation}
  be a commutative diagram in $\largecats$ such that, and suppose that $h$ admits a right adjoint $h^R$ for each $h\in\lbrace f,i,i',g\rbrace$.
  Assume that $i$ and $i'$ are both fully faithful.
  If $\Image (g^Ri')\subseteq\Image (i)$, then~\eqref{eq:fully-faithfullness-right-adjointability} is horizontally right adjointable.
\end{lem}
\begin{proof}
  Equivalently, the transpose square
  \begin{equation*}
    \begin{tikzcd}
      \mathcal D'\arrow[r,"f^R"]\arrow[d,"(i')^R"]
      & \mathcal D\arrow[d,"i^R"] \\
      \mathcal C'\arrow[r,"g^R"]
      & \mathcal C
    \end{tikzcd}
  \end{equation*}
  is vertically left adjointable~\cite[Rem~4.7.4.14]{Lurie2017}. The Beck--Chevalley morphism is given by
  \[
    ig^R\to ig^R(i')^Ri'\simeq ii^Rf^Ri'\xrightarrow\sim f^Ri',
  \]
  where the first map is the unit of the $i'\dashv (i')^R$ adjunction and the final map is the counit of the $i\dashv i^R$ adjunction.
  The latter is an equivalence since $i'$ is fully faithful.
  The former is an equivalence if $g^Ri'$ factors through $i$, i.e. if the essential image of $g^Ri'$ is contained in the essential image of $i$.
\end{proof}
\begin{prop}\label{prop:beck-chevalley-for-homd}
  Let $\mathcal C$ and $\mathcal C'$ be proper and $\omega_1$-compact dualizable stable categories, and let $\mathcal D$ be some dualizable stable category.
  If $f\colon\mathcal C\to\mathcal C'$ is strongly continuous, then
  \begin{equation}
    \label{eq:right-adjointability}
    \begin{tikzcd}[column sep=large]
      \Homd (\mathcal C',\mathcal D)^\vee\arrow[r,"((f^*)^R)^\vee"]\arrow[d,"\Phi"]
      & \arrow[d,"\Phi"] \Homd (\mathcal C,\mathcal D)^\vee\\
      \Ind(\Fun^L (\mathcal C',\mathcal D)^{\mathrm{op}})\arrow[d,phantom,"\rotatebox{270}{$\simeq$}"]\arrow[r,"\Ind ({-}\circ f^{\mathrm{op}})"]
      & \Ind (\Fun^L (\mathcal C,\mathcal D)^{\mathrm{op}})\arrow[d,phantom,"\rotatebox{270}{$\simeq$}"]\\[-1em]
      \Fun ((\mathcal C')^\vee\otimes\mathcal D,\Sp)\arrow[r]
      & \Fun ((\mathcal C')^\vee\otimes\mathcal D,\Sp)
    \end{tikzcd}
  \end{equation}
  is horizontally right adjointable.
\end{prop}
\begin{proof}
  We check the conditions of~\cref{lem:beck-chevalley}. The argument is similar to the proof of~\cite[Prop~3.13]{Efimov2025b}, and we use the notation from there.
  The image of
  \[
    \Homd (\mathcal C,\mathcal D)\xrightarrow\Phi\Ind (\Fun^L (\mathcal C,\mathcal D))\simeq\Ind (\mathcal C^\vee\otimes\mathcal D)
  \]
  is generated by the objects that Efimov denotes $\operatorname{THC}(\mathcal C,\mathrm{ev}_{\mathcal D}({-},M))$, see~\cite[Prop~3.12]{Efimov2025b}, for $M\in (\mathcal C\otimes\mathcal D^\vee )^{\omega_1}$. The right adjoint to the bottom horizontal arrow in~\eqref{eq:right-adjointability} is precomposition with $(f^R)^\vee\boxtimes\id\colon (\mathcal C')^\vee\otimes\mathcal D\to\mathcal C^\vee\otimes\mathcal D$, which takes $\operatorname{THC}(\mathcal C,\mathrm{ev}_{\mathcal D}({-},M))$ to $\operatorname{THC}(\mathcal C',\mathrm{ev}_{\mathcal D}({-},((f^R)^\vee\boxtimes\id ) M))$. But the latter belongs to the image of $\Phi$ by~\cite[Prop~3.12]{Efimov2025b}. Since precomposition with $(f^R)^\vee\boxtimes\id$ preserves colimits, this finishes the proof.
\end{proof}
\begin{proof}[Proof of~\cref{thm: codescent}]
  The inclusion $\Prdual\hookrightarrow\Prlst$ preserves and reflects colimits, so it suffices to show that~\eqref{eq:codescent-diagram} is a colimit in the latter category. Equivalently, the image of~\eqref{eq:codescent-diagram} under the {\em contravariant} involution $(\mathcal C\mapsto\mathcal C^\vee,F\mapsto F^\vee )$, namely the diagram
  \begin{equation}
    \label{eq:alt-codescent-diagram}
    \begin{tikzcd}
      \ccoShv{Y}^\vee\arrow[r]\arrow[d]
      & \ccoShv{Y_{1}}^\vee \arrow[d] \\
      \ccoShv{Y_{0}}^\vee\arrow[r]
      & \ccoShv{Y_{01}}^\vee,
    \end{tikzcd}
  \end{equation}
  is a pullback diagram in $\Prlst$, or equivalently in the (very large) category of large categories $\largecats$. 
  By~\cref{prop:beck-chevalley-for-homd}, the diagram~\eqref{eq:alt-codescent-diagram} is the back face in a cubical commutative diagram in the category of large categories
 \[
   \begin{small}
     \begin{tikzcd}[column sep=tiny]
       \ccoShv{Y}^\vee\arrow[rd] \arrow[rr]\arrow[dd]&& \ccoShv{Y_1}^\vee\arrow[rd]\arrow[dd] \\ 
       & \Fun (\coShv{Y}^{\omega_1},\Sp )\arrow[rr]
       \arrow[dd]
       && \Fun(\coShv{Y_1}^{\omega_1},\Sp )\arrow[dd]\\
       \ccoShv{Y_0}^\vee\arrow[rd] \arrow[rr]&& \ccoShv{Y_{01}}^\vee\arrow[rd] \\
       & \Fun (\coShv{Y_0}^{\omega_1},\Sp )\arrow[rr]
       && \Fun (\coShv{Y_{01}}^{\omega_1},\Sp ),
     \end{tikzcd}
   \end{small}
\]
in which
\begin{enumerate}[label=(\roman*)]
\item For each $Z\in\lbrace Y,Y_0,Y_1,Y_{01}\rbrace$, the functor
  \[
    \ccoShv Z^\vee\to\Fun (\coShv{Z}^{\omega_1},\Sp )\simeq\Ind ((\coShv{Z}^{\omega_1})^{\mathrm{op}})
  \]
  is fully faithful, and its essential image is generated under colimits by the objects that Efimov denotes $\operatorname{THC}(\Shv Z /\Sp ,\mathrm{ev}_{\Sp /\Sp }({-},F) )$ for $F\in\Shv Z^{\omega_1}$, see~\cite[Prop~3.12]{Efimov2025b}. In our setting, we find that for $F$ as before
  \[
    \operatorname{THC}(\Shv Z /\Sp ,\mathrm{ev}_{\Sp /\Sp }({-},F) )\colon G\mapsto G(F).
  \]
\item For $Z$ and $W\in\lbrace Y,Y_0,Y_1,Y_{01}\rbrace$ with inclusion $\iota\colon Z\hookrightarrow W$, the functor
  \[
    \Fun (\coShv{W}^{\omega_1},\Sp)\to\Fun (\coShv{Z}^{\omega_1},\Sp )
  \]
  is ${-}\circ (\iota^*)^\vee$. This functor sits inside a triple of adjunctions
  \[
    \Ind (\iota_*^\vee)\dashv \Ind ((\iota^*)^\vee)\dashv ({-})\circ (\iota^*)^\vee.
  \]
  Passing to left adjoints twice in the front face of the cubical diagram, we thus get the square
  \[
    \begin{tikzcd}[column sep=large]
       \Fun (\coShv{Y}^{\omega_1},\Sp )\arrow[r,"\Ind ((\iota_1)_*^\vee)"]
       \arrow[d,"\Ind ((\iota_1)_*^\vee)"] 
       & \Fun(\coShv{Y_1}^{\omega_1},\Sp )\arrow[d,"\Ind ((\iota_{10})_*^\vee)"]\\
       \Fun (\coShv{Y_0}^{\omega_1},\Sp )\arrow[r,"\Ind ((\iota_{01})_*^\vee)"]
       & \Fun (\coShv{Y_{01}}^{\omega_1},\Sp ),
    \end{tikzcd}
  \]
  This square exhibits
  $\Fun (\coShv{Y}^{\omega_1},\Sp )$ as the pullback
  \[
    \Fun (\coShv{Y_0}^{\omega_1},\Sp )\times_{\Fun (\coShv{Y_{01}}^{\omega_1},\Sp )}^{\mathrm{dual}} \Fun (\coShv{Y_1}^{\omega_1},\Sp )
  \]
  in the category of dualizable categories, since $\Ind ({-}^{\omega_1})\colon\Prlstuncount\to\Prdual$ is a right adjoint~\cite[Prop~1.89]{Efimov2025}
  and
  \[
    \begin{tikzcd}[column sep=large]
      \coShv{Y_{01}}\arrow[r,"((\iota_{10}^*)^R)^\vee"]\arrow[d,"((\iota_{01}^*)^R)^\vee"]
      & \coShv{Y_{1}} \arrow[d,,"((\iota_{1}^*)^R)^\vee"] \\
      \coShv{Y_{0}}\arrow[r,"((\iota_{0}^*)^R)^\vee"]
      & \coShv{Y},
    \end{tikzcd}
  \]
  is a pullback in $\Prlstuncount$ by~\cite[Prop~2.31]{Ramzi2024}. In particular, \cite[Prop~1.87]{Efimov2025} implies that the induced functor from $\Fun (\coShv{Y}^{\omega_1},\Sp )$ to the pullback \[\Fun (\coShv{Y_0}^{\omega_1},\Sp )\times_{\Fun (\coShv{Y_{01}}^{\omega_1},\Sp )}\Fun (\coShv{Y_1}^{\omega_1},\Sp )\] computed in $\Prlst$ (or equivalently in $\largecats$) is fully faithful and strongly continuous. Passing to right adjoints twice and using the fact that the two-fold right adjoint of a fully faithful functor is again fully faithful, we find that in the front face of the cubical diagram, the canonical functor from the upper left corner to the pullback (computed in $\largecats$) of the bottom horizontal and right vertical functors is fully faithful.
\end{enumerate}
It follows from points (i) and (ii) that we can identify \(\ccoShv{Y_0}\times_{\ccoShv{Y_{01}}}\ccoShv{Y_1}\) with the full subcategory of $\Fun (\coShv{Y}^{\omega_1},\Sp )$ spanned by functors $G$ such that $G\circ (\iota^*)^\vee\simeq G\circ (\iota)_*$ lies in the essential image of $\ccoShv{Z}^\vee\to\Fun (\coShv{Z}^{\omega_1},\Sp )$ for $Z\in\lbrace Y_0,Y_1,Y_{01}\rbrace$ with inclusion $\iota\colon Z\hookrightarrow Y$. But the latter is equivalent to $G\circ \iota_*\simeq\colim_{j\in J(Z)}\mathrm{ev}_{\Shv {Z}}({-},F_{Z,k})$, where each $F_{Z,k}$ lies in $\Shv{Z}^{\omega_1}$. But for such $G$ we find a pushout
\[
  \begin{tikzcd}
    G\arrow[r]\arrow[d]
    & \left(\colim_{j\in J(Y_1)}\mathrm{ev}_{\Shv {Y_1}}({-},F_{j,Y_1})\right)\circ\iota_1^*\arrow[d]\\
    \left(\colim_{j\in J(Y_0)}\mathrm{ev}_{\Shv {Y_0}}({-},F_{j,Y_0})\right)\circ\iota_0^*\arrow[r]
    & \left(\colim_{j\in J(Y_{01})}\mathrm{ev}_{\Shv {Y_{01}}}({-},F_{j,Y_{01}})\right)\circ (\iota_0\iota_{01})^*,
  \end{tikzcd}
\]
and we find also that
\[
  \left(\colim_{j\in J(Z)}\mathrm{ev}_{\Shv {Z}}({-},F_{j,Z})\right)\circ \iota^*
\simeq   \colim_{j\in J(Z)}\mathrm{ev}_{\Shv {Z}}({-},\iota_*F_{j,Z}),
\]
which belongs to the essential image of
\[
    \ccoShv{Y}^\vee\to\Fun (\coShv{Y}^{\omega_1},\Sp )\simeq\Ind ((\coShv{Y}^{\omega_1})^{\mathrm{op}})
  \]
  by~\cite[Prop~3.13]{Efimov2025b}, finishing the proof.
\end{proof}
\subsection{Proof of~\cref{intro epimorphism}}
We can now prove~\cref{intro epimorphism}, or rather the more general~\cref{thm: epimorphism}.
\begin{proof}[Proof of~\cref{thm: epimorphism}]
  Assume first that $X$ is contractible. In this case, we claim that the assembly functor is equivalent to the functor
  \[
    \ccoShv X = \Homd (\Shv X,\Sp )\to \Homd (\Sp,\Sp )\simeq\Sp
  \]
  that is induced by the strongly continuous functor $X^*\colon\Sp\to\Shv X$. Indeed, the map $\psi_X^*\colon\Sp\simeq\Sp^{\type X}\to\Shv X$ identifies with $X^*$, so the Bartels--Nikolaus copairing $\delta$ is in this case left adjoint to
  \[
    \begin{tikzpicture}
      \node (V0) at (0,0){$\Shv X$};
      \node (V1) at (3,0){$\Shv X\otimes\Sp$};
      \node (V2) at (7,0){$\Shv{X\times X}$};
      \node (V3) at (10,0){$\Shv X$};
      \node (V4) at (12.5,0){$\Sp$.};

      \draw[->,draw=none] (V0) -- (V1) node[midway]{$\simeq$};
      \draw[->] (V1) to node[midway,above]{$\id\otimes X^*$} (V2);
      \draw[->] (V2)to node[midway,above]{$\Delta_X^*$} (V3);
      \draw[->] (V3) to node[midway,above]{$X_*$} (V4);
      \draw[->] (V0) to [out=340,in=200]  node[midway,above]{$\id$} (V3);
    \end{tikzpicture}
  \]
  Here the composition of the first three arrows is canonically identified with $\Delta_X^*(\id\times X)^*\simeq (\Delta_X\circ\id\times X)^*$, and $\Delta_X\circ\id\times X=\id$. Hence $\delta$ is left adjoint to $X_*$, proving that $\delta$ is given by $X^*$, and by construction the Bartels--Nikolaus assembly functor then identifies with the functor induced by $X^*$. We have adjunctions
  \[
    X_\sharp\dashv X^*\dashv X_*\dashv X^!,
  \]
  and contractibility of $X$ implies that $X^*$ is fully faithful. The latter means that $X_\sharp X^*\simeq\id$. Since both $X_\sharp$ and $X^*$ are strongly continuous, this formula witnesses that $X^*$ is a split monomorphism in the category of dualizable categories. Since the Bartels--Nikolaus assembly functor $\BNass$ is contravariantly induced by $X^*$, it follows that $\BNass$ is a split epimorphism, and in particular that it is a reflective localization.

  Now consider an arbitrary compact ANR $X$ with a decent cover $\lbrace X_i\rbrace_{i\in I}$. We consider the power set $\mathcal P(I)$ as a poset under reverse inclusion, and we let $\mathcal P_0$ denote the subposet spanned by finite subsets $J\subseteq I$ such that $\bigcap_{j\in J}X_j$ is non-empty.
  For each such $J$, we let $X_J = \bigcap_{j\in J}X_j$ and $K_J = \type{X_J}$. We have a natural transformation
  \begin{equation}
    \label{eq:nat-transform-of-functors}
    \ccoShv{X_J}\to\Sp^{K_J}
  \end{equation}
  of functors from $\mathcal P_0$ into the category of dualizable stable categories, given pointwise by the Bartels--Nikolaus assembly functor.
  By the case we have already treated, this natural transformation is pointwise a reflective localization.
  But then the top horizontal functor in the diagram
  \[
    \begin{tikzcd}
      \colim_{J\in\mathcal P_0}\ccoShv{X_J}\arrow[r]\arrow[d]
      & \colim_{J\in\mathcal P_0}\Sp^{K_J}\arrow[d]\\
      \ccoShv{X}\arrow[r]
      & \Sp^{K}
    \end{tikzcd}
  \]
  is a reflective localization. Here the vertical arrows are induced by the pushforwards $\ccoShv{X_J}\to\ccoShv X$ and $\Sp^{K_J}\to\Sp^K$.
  The left vertical arrow is an equivalence by~\cref{thm: codescent}, and the right vertical arrow is an equivalence by a version of the nerve theorem~\cite[Cor~1.5]{Dugger_Isaksen2004}. We conclude that the Bartels--Nikolaus assembly functor for $X$ is a reflective localization as claimed.
\end{proof}


\section{The torsion cosheaf of a homotopy equivalence}
\label{sec: htpy eq}
The heart of this article is the following construction:
\begin{cons}\label{cons:torsion-cosheaf}
  Let $f\colon Y\to X$ be a continuous map between compact ANRs. There is an induced map
  \begin{equation}
    \label{eq:map-of-wall-objects}
    f_?\WallObject Y\to \WallObject X,
  \end{equation}
  defined in terms of the corresponding cosheaves as the composite
  \[
    \begin{tikzcd}
      Y_\sharp f^*\arrow[r,"\text{unit}"]
      & Y_\sharp f^*X^*X_\sharp
      \arrow[rr,"\text{functoriality}"above, "\sim"below]
      & & Y_\sharp Y^*X_\sharp \arrow[r,"\text{counit}"]
      & X_\sharp .
    \end{tikzcd}
  \]
  We define the \tdef{torsion cosheaf} of $f$, denoted $\mdef{\Tors f}$, by the formula
  \[
    \Tors f = \cofiber (f_?\WallObject Y\to\WallObject X)\in \ccoShv X.
  \]
\end{cons}
Our main task in this section is to show that $\Tors f$ recovers the Whitehead torsion of yore. This is the content of~\cref{subsec: K-theory class} below, after some preliminaries in~\cref{subsec:rel-whitehead} and~\cref{subsec: recollection on whitehead}. The remaining sections give some basic properties of and intuition about $\Tors f$.
\begin{rmk}
  Unfolding the construction, we see that the torsion cosheaf $\Tors f\in\ccoShv X$ is associated with the cosheaf
  \[
    U\mapsto\cofiber (\suspinf{\type{f^{-1}(U)}}\to\suspinf{\type{U}}),
  \]
  where $\suspinf{\type{f^{-1}(U)}}\to\suspinf{\type{U}}$ is induced by the restriction of $f$ to $f^{-1}(U)$. In other words, $\Tors f$ measures the extent to which $f$ is a homology isomorphism locally on $X$.
\end{rmk}

\subsection{Relation to the Whitehead category}\label{subsec:rel-whitehead}
The goal of this section is to show
\begin{thm}\label{thm:torsion-belongs-to-whitehead}
  If $f\colon Y\to X$ is a map between compact ANRs that induces isomorphisms on all homotopy groups, then the torsion cosheaf $\Tors f$ is a compact object in the Whitehead category $\WhiteheadCat X\subseteq\ccoShv X$.
\end{thm}
The theorem is an almost immediate consequence of the following lemma, which we will have other uses for later:
\begin{lem}\label{lem:induced-map-on-parametrized-spectra}
  Let $f\colon Y\to X$ be as in~\cref{cons:torsion-cosheaf}. The Bartels--Nikolaus assembly functor takes~\eqref{eq:map-of-wall-objects} to the canonical map of parametrized spectra
  \[
    f_\sharp\SS_{\type Y}\simeq f_\sharp f^*\SS_{\type X}\xrightarrow{\mathrm{counit}}\SS_{\type X}.
  \]
\end{lem}
\begin{proof}
  Let $A\in\ccoShv X$ be a compact object, or equivalently a strongly continuous functor $A\colon\Sp\to\ccoShv X$. The corresponding cosheaf $\widetilde A\in\Shv X$ is defined by the following string diagram
  \begin{equation}\label{eq:cosheaf-from-cpt-object}
    \begin{tikzpicture}[baseline=(current bounding box.center)]
      \node (A) at (-1,1.25) [rectangle, draw, minimum width=0.6cm, minimum height=0.6cm] {$A$};
      \node (ev) at (0,0) [rectangle, draw, minimum width=0.6cm, minimum height=0.6cm] {$\mathrm{ev}$};
      \node (shvXLHS) at (-5,2) {$\Shv X$};
      \node (tilA) at (-5,0)[rectangle, draw, minimum width=0.6cm, minimum height=0.6cm] {$\widetilde A$};
      \node (shvXRHS) at (1,2){$\Shv X$};
      \node (equals) at (-2.5,0.5){$=$};

      \draw (shvXLHS) -- (tilA);
      \draw (A.south) to [out=270, in=90] (ev.north);
      \draw (shvXRHS.south) to [out=270, in=90](ev.north);
    \end{tikzpicture}
  \end{equation}
  Using this formula, we compute the value of $A$ under the Bartels--Nikolaus assembly functor
  \begin{equation}\label{eq:value_under_assembly}
    \begin{tikzpicture}[baseline=(current bounding box.center)]
      \node (A) at (-1,3.5) [rectangle, draw, minimum width=0.6cm, minimum height=0.6cm] {$A$};
      \node (ev) at (0,0) [rectangle, draw, minimum width=0.6cm, minimum height=0.6cm] {$\mathrm{ev}$};
      \node (delta) at (1,2) [rectangle, draw, minimum width=0.6cm, minimum height=0.6cm] {$\delta$};
      \node (spX) at (2,-1){$\Sp^{K}$};

      \draw (A.south) to [out=270, in=90] (ev.north);
      \draw (delta.south) to [out=270, in=90](ev.north);
      \draw (delta.south) to [out=270, in=90](spX.north);
    \end{tikzpicture}
    =\quad
    \begin{tikzpicture}[baseline=(current bounding box.center)]
      \node (A) at (-1,2) [rectangle, draw, minimum width=0.6cm, minimum height=0.6cm] {$A$};
      \node (ev) at (0,0) [rectangle, draw, minimum width=0.6cm, minimum height=0.6cm] {$\mathrm{ev}$};
      \node (delta) at (1,3.5) [rectangle, draw, minimum width=0.6cm, minimum height=0.6cm] {$\delta$};
      \node (spX) at (2,-1){$\Sp^{K}$};

      \draw (A.south) to [out=270, in=90] (ev.north);
      \draw (delta.south) to [out=270, in=90](ev.north);
      \draw (delta.south) to [out=270, in=90](spX.north);
    \end{tikzpicture}
    =\quad
    \begin{tikzpicture}[baseline=(current bounding box.center)]
      \node (tilA) at (0,0) [rectangle, draw, minimum width=0.6cm, minimum height=0.6cm] {$\widetilde{A}$};
      \node (delta) at (1,3.5) [rectangle, draw, minimum width=0.6cm, minimum height=0.6cm] {$\delta$};
      \node (spX) at (2,-1){$\Sp^{K}$,};

      \draw (delta.south) to [out=270, in=90](ev.north);
      \draw (delta.south) to [out=270, in=90](spX.north);
    \end{tikzpicture}
  \end{equation}
  where we have put $K = \type X$.
  For the final equality, we used~\eqref{eq:cosheaf-from-cpt-object}.
  We thus have a natural identification of $\BNass (A)$ with the composite
  \[
    \Sp\xrightarrow{\delta}\Shv X\otimes\Sp^{K}\xrightarrow{\widetilde A\otimes\id}\Sp^{K},
  \]
  where $\delta$ is the copairing constructed by Bartels and Nikolaus. That is, $\delta$ is left adjoint to
  \[
    \Shv X\otimes\Sp^K\xrightarrow{\id\otimes\psi_X^*}\Shv{X\times X}\xrightarrow{\Delta_X^*}\Shv X\xrightarrow{X_*}\Sp.
  \]
  Taking $A = \WallObject X$ in~\eqref{eq:value_under_assembly} and recalling that $\widetilde A = X_\sharp$, we then find that $\BNass (\WallObject X )\colon\Sp\to\Sp^K$ is right adjoint to
  \begin{align}\label{eq:ident_of_wall}
    \begin{split}
      &X_*\Delta_X^*(\id_{\Shv X}\otimes\psi_X^*)(\psi_X^*\otimes\id_{\Sp^K} )(K^*\otimes\id_{\Sp^K})\\
      &\quad\simeq X_*\Delta_X^*(\psi_X^*\otimes\psi_X^*)(K^*\otimes\id_{\Sp^K})\\
      &\quad\simeq X_*\psi_X^*(\Delta_K)^*(K^*\otimes\id_{\Sp^K})\\
      &\quad\simeq K_*(\psi_X)_*\psi_X^*(\Delta_K)^*(K^*\otimes\id_{\Sp^K})\\
      &\quad\simeq K_*(\Delta_K)^*(K^*\otimes\id_{\Sp^K})\\
      &\quad\simeq K_*.
    \end{split}
  \end{align}
  Here the first three equivalences are functoriality, the fourth is the counit $(\psi_X)_*\psi_X^*\to\id$ (which is an equivalence since $\psi_X^*$ is fully faithful), and the fourth is functoriality again since $(K\times\id)\circ\Delta_K=\id$. Passing back to left adjoints, we recover Bartels and Nikolaus's calculation that $\BNass (\WallObject X )$ is $K^* = \SS_K$.

  Similarly, taking $A = f_?\WallObject Y$ and putting $L = \type Y$, we can rewrite the right adjoint of $\BNass (f_?\WallObject Y)\colon\Sp\to\Sp^X$ as
  \begin{align}\label{eq:ident_of_pushforward_wall}
    \begin{split}
      &X_*\Delta_X^*(\id_{\Shv X}\otimes\psi_X^*)(f_*\otimes\id )(\psi_Y^*\otimes\id_{\Sp^K} )(L^*\otimes\id_{\Sp^K})\\
      &\quad\simeq X_*\Delta_X^*(f_*\otimes\id )(\id_{\Shv X}\otimes\psi_X^*)(\psi_Y^*\otimes\id_{\Sp^K} )(L^*\otimes\id_{\Sp^K})\\
      &\quad\simeq X_*f_*\Delta_Y^*(\id_{\Shv X}\otimes f^*)(\psi_Y^*\otimes\psi_X^*)(L^*\otimes\id_{\Sp^K})\\
      &\quad\simeq X_*f_*\Delta_Y^*(\psi_Y^*\otimes\psi_Y^*)(\id_{\Sp^L}\otimes f^*)(L^*\otimes\id_{\Sp^K})\\
      &\quad\simeq L_*(\psi_Y)_*\psi_Y^*\Delta_L^*(\id_{\Sp^L}\otimes f^*)(L^*\otimes\id_{\Sp^K})\\
      &\quad\simeq L_*(\psi_Y)_*\psi_Y^*\Delta_L^*(L^*\otimes\id_{\Sp^L})(\id_{\Sp}\otimes f^*)\\
      &\quad\simeq L_*\Delta_L^*(L^*\otimes\id_{\Sp^L})(\id_{\Sp}\otimes f^*)\\
      &\quad\simeq L_*f^*.
    \end{split}
  \end{align}
  Here the first equivalence is functoriality, the second is the Beck--Chevalley morphism associated with the pullback
  \[
    \begin{tikzcd}
      Y\arrow[d,"\Delta_Y"]\arrow[r,"f"]
      & X\arrow[dd,"\Delta_X"]\\
      Y\times Y \arrow[d,"\id\times f"]\\
      Y\times X\arrow[r,"f\times\id"]
      & X\times X,
    \end{tikzcd}
  \]
  the third and fourth equivalences are functoriality, the fifth is the counit $(\psi_Y)_*\psi_Y^*\to\id$, and the final equivalence is functoriality. Passing back to left adjoints, we get an equivalence $\BNass (f_?\WallObject Y) = f_\sharp\SS_L$ as claimed.

  The map $f_?\WallObject Y\to\WallObject X$ is the Beck--Chevalley morphism associated with taking left adjoints of the vertical edges of the outer rectangle in the commutative diagram
  \begin{equation}
    \label{eq:composition-of-bc-morphisms}
    \begin{tikzcd}
      \Shv X\arrow[r,"f^*"]
      &\Shv Y\\
      \Sp^{K}\arrow[r,"f^*"]\arrow[u,"(\psi_X)^*"]
      & \Sp^{L}\arrow[u,"(\psi_Y)^*"]\\
      \Sp\arrow[r,equals]\arrow[u," K^*"]
      &\Sp.\arrow[u,"L^*"]
    \end{tikzcd}
  \end{equation}
  So as to compare with the calculations~\eqref{eq:ident_of_wall} and~\eqref{eq:ident_of_pushforward_wall} (in which we passed to right adjoints), we must consider the transpose Beck--Chevalley morphism. This is the Beck--Chevalley morphism associated with taking right adjoints of the horizontal arrows in the outer rectangle of~\eqref{eq:composition-of-bc-morphisms}. Since Beck--Chevalley morphisms compose, we can also write the Beck--Chevalley morphism associated with the outer rectangle in~\eqref{eq:composition-of-bc-morphisms} as the composite
  \begin{equation*}
    \begin{tikzpicture}
      \node (V0) at (0,0){$(\psi_X)^*K^*$};
      \node (V1) at (2.5,0){$(\psi_X)^*f_*f^*K^*$};
      \node (V2) at (4.9,0){$(\psi_X)^*f_*L^*$};
      \node (V3) at (7.75,0){$f_*f^*(\psi_X)^*f_*L^*$};
      \node (V4) at (10.6,0){$f_*(\psi_Y)^*f^*f_*L^*$};
      \node (V5) at (13.4,0){$f_*(\psi_Y)^*L^*$,};

      \draw[->] (V0)--(V1);
      \draw[->,draw=none] (V1)--(V2) node[midway]{$\simeq$};
      \draw[->] (V2)--(V3);
      \draw[->,draw=none] (V3)--(V4) node[midway]{$\simeq$};
      \draw[->] (V4)--(V5);
    \end{tikzpicture}
  \end{equation*}
  where the equivalences are functoriality, the first arrow is the unit for the $f^*\dashv f_*$ adjunction for parametrized spectra, the second arrow is the unit for the $f^*\dashv f_*$ adjunction for sheaves, and the third arrow is the counit for the $f^*\dashv f_*$ adjunction for parametrized spectra. Since both the identifications~\eqref{eq:ident_of_wall} and~\eqref{eq:ident_of_pushforward_wall} and the morphism $(\psi_X)^*K^*\to f_*(\psi_Y)^*L^*$ are defined entirely in terms of functoriality coherence maps and unit/counit transformations, it is then an exercise in using that (co)unit transformations compose to show that $\BNass (f_?\WallObject Y\to\WallObject X)$ is identified (after passing to right adjoints) with the unit transformation $K^*\to K^*f^*f_*\simeq L^*f_*$. Passing back to left adjoints, this is precisely the map $f_\sharp\SS_L\to\SS_K$ as claimed.
\end{proof}
\begin{proof}[Proof of~\cref{thm:torsion-belongs-to-whitehead}]
  Note that both $f_?\WallObject Y$ is a compact object in $\WhiteheadCat X$ since $f_?$ is strongly continuous. Hence $\Tors f$ is the cofiber of a map between compact objects in $\ccoShv X$, and hence itself a compact object of the latter. By \cref{lem:induced-map-on-parametrized-spectra}, we find that $\BNass (f_?\WallObject Y\to\WallObject X)$ identifies with the counit transformation $f_\sharp f^*\SS_{\type X}\to\SS_{\type X}$, which is an equivalence since $f^*\colon\Sp^{\type X}\to\Sp^{\type Y}$ is an equivalence. Hence $\Tors f$ lies in $\WhiteheadCat X$. The inclusion $\WhiteheadCat X\hookrightarrow\ccoShv X$ is fully faithful and strongly continuous, whence it reflects compactness, finishing the proof.
\end{proof}

\subsection{Recollection on Whitehead torsion}\label{subsec: recollection on whitehead}
Suppose $f\colon Y\to X$ is a map between compact ANRs, and that $X$ admits a decent cover. Since $\Tors f$ is a compact object in $\WhiteheadCat X$, it descends to a class in K-theory $\lbrack \Tors f\rbrack\in\K_0(\WhiteheadCat X)\simeq\WhiteheadSpt{\type X}[1]$. In the next subsection we will show that this class recovers the classically defined Whitehead torsion of $f$, but we start by recalling how this class is constructed for the reader's convenience. We closely follow the exposition in~\cite{Lurie2014}.

With $f\colon Y\to X$ as above, we start by choosing a CW structure on $X$ together with a CW approximation of $f$; that is, a homotopy commutative triangle
\[
  \begin{tikzcd}
    Y\arrow[r,"f"]\arrow[d,"g"]
    & X\\
    Y'.\arrow[ru,swap,"f'"]
  \end{tikzcd}
\]
in which $Y'$ is a CW complex, $g$ is a weak homotopy equivalence, and $f'$ is a cellular map. We will assume that $X$ is connected and leave it to the reader to extend to the disconnected case. The map $f'$ lifts to a map $\tilde f'\colon \widetilde{Y'}\to \widetilde{X}$ of universal covers, which is unique up to a choice of basepoint. By the homotopy lifting property, there are unique cell structures on $\widetilde{Y'}$ and $\widetilde{X}$ such that $\tilde f'$ as well as the projections $\widetilde{X}\to X$ and $\widetilde{Y'}\to Y'$ are all cellular maps. The map $f'$ induces a map on cellular chains
\[
\tilde f'_*\colon\cchains{\widetilde Y'}\to\cchains{\widetilde X}.
\]
Note that both chain complexes admit an action of $\pi = \uppi_1X$ by Deck transformations, and that $\tilde f'_*$ is $\pi$-equivariant. We let $C$ denote (the usual model for) the mapping cone of $\tilde f'_*$; that is, $C$ has chains
\[
C^{\vphantom{\mathrm{cw}}}_i = \cchains{\widetilde{Y'}}[i-1]\oplus \cchains{\widetilde X}[i],
\]
and boundary map
\[
  d_i =
  \begin{pmatrix}
    -d'_{i-1} & 0 \\
    f'_* & d_i'',
  \end{pmatrix}
\]
where we have used $d'$ and $d''$ to denote the boundary maps of $\cchains{\widetilde Y'}$ and $\cchains{\widetilde{Y'}}$ respectively.

Since $\widetilde f'_*$ is a quasi-isomorphism and $C$ is a model for its mapping cone, we find that $C$ is acyclic. This gives a path from $C$ to $0\in\upOmega^\infty\K (\ZZ\pi )$.

The chain complex $C$ has a preferred basis as $\cchains{X}$ and $\cchains{ Y'}$ are defined to be free on the cells of $X$ and $Y'$ respectively, and there are canonical isomorphisms $\cchains{\widetilde X}\simeq\ZZ\pi\otimes\cchains{X}$ and $\cchains{\widetilde{Y'}}\simeq\ZZ\pi\otimes\cchains{Y'}$. This gives a path from $C$ to
\[
  \sum_{e\in \mathrm{cells}(X)}\susp^{\dim e}\ZZ\pi + \sum_{e'\in \mathrm{cells}(Y')} \susp^{\dim (e')-1}\ZZ\pi
\]
in $\desusp^\infty\K(\ZZ\pi )$.
Let
\begin{align*}
  E&=\lbrace
     e\in\mathrm{cells}( X)\mid \dim e\text{ is even}
     \rbrace
     \cup
     \lbrace
     e\in\mathrm{cells}({Y'})\mid \dim e\text{ is odd}
     \rbrace,\quad\text{and}\\
  E' &= \lbrace
  e\in\mathrm{cells}( X)\mid \dim e\text{ is odd}
  \rbrace
  \cup
  \lbrace
  e\in\mathrm{cells}({Y'})\mid \dim e\text{ is even}
    \rbrace.
\end{align*}
Since $ f'$ is a homotopy equivalence, we get that $\eulerchar ( X) = \eulerchar ({Y'})$. Since the Euler characteristic of a finite CW complex is equal to the number of even-dimensional cells minus the number of odd-dimensional cells, we thus get that $\# E = \# E'$. Pick a bijection $\phi$ between $E$ and $E'$.
Given $e\in E\cup E'$, we let
\[
  d(e) = \begin{cases}
    \dim e,&\text{if } e\in\mathrm{cells}(X),\\
    \dim e -1,&\text{if } e\in\mathrm{cells}(Y).
  \end{cases}
\]
Then $\ZZ\pi\lbrack d(e)\rbrack$ and $\ZZ\pi\lbrack d(\phi (e))\rbrack$ are isomorphic modules in degrees of complementary parity, we get a path from
\[
\susp^{d(e)}\ZZ\pi  + \susp^{d(\phi (e))}\ZZ\pi 
\]
to $0\in\desusp^\infty\K (\ZZ\pi )$. Summing over all these paths for $e$ varying over $E$, we get a path from $C$ to $0$. Joining this with the path we got from the acyclicity of $C$, we get a loop $\gamma\in\desusp^{\infty +1}\K(\ZZ\pi )$. This descends to a class $\lbrack\gamma\rbrack\in\K_1(\ZZ\pi )$. Now we invoke
\begin{fact}\label{fact:low-degs}
  For a connected finitely-dominated $\infty$-groupoid $K$, the base change map \[\A_i(K) = \K_i(\SS\lbrack\upOmega K\rbrack)\to \K_i(\ZZ\lbrack \uppi_1K\rbrack)\] is an isomorphism for $0\leq i\leq 1$.
\end{fact}
This is a special case of the well-known fact that if $R$ is a connective associate ring spectrum, then $\K_i(R)\to\K_i(\uppi_0R)$ is an isomorphism for $i=0,1$ (see Corollaries 3 and 4 in Lecture 20 of~\cite{Lurie2014}).

It follows from the above that we get a unique lift $\lbrack \tilde\gamma\rbrack\in\K_1(\SS\lbrack\upOmega K\rbrack )$ of the class $\lbrack\gamma\rbrack$, where $K = \type X$.
\begin{defn}
  The \tdef{Whitehead torsion} of $f$, denoted $\mdef{\uptau (f)}$, is the image of the class $\lbrack \tilde\gamma\rbrack\in\K_1(\SS\lbrack\upOmega K\rbrack )$ under the map
  \[
    \K_1(\SS\lbrack\upOmega K\rbrack ) \to\WhiteheadSpt{K}[1].
  \]
\end{defn}
\subsection{The K-theory class of the torsion cosheaf}\label{subsec: K-theory class}
The goal of this subsection is to prove
\begin{thm}\label{thm:torsion-is-torsion}
  If $f\colon Y\to X$ is a simplicial map between compact triangulated spaces,
  then the class $\lbrack\Tors f\rbrack\in \K_0(\WhiteheadCat X)$ is equal to the Whitehead torsion of $f$ under the canonical identification $\K_0(\WhiteheadCat X)\simeq\WhiteheadSpt{\type X}[1]$.
\end{thm}
Our proof is modelled on the proof of a similar statement in Lecture 27 of~\cite{Lurie2014}. As in the latter proof, we will need the following fact about the Waldhausen assembly map, which follows from~\cref{fact:low-degs}:
\begin{fact}\label{fact:monomorphism}
  For a connected finitely-dominated $\infty$-groupoid $K$, the mapping
  \[
    \uppi_0(\A (\pt )\otimes K)\to\uppi_0\A (K) = \A_0(K)
  \]
  is a monomorphism.
\end{fact}
We will also need the following lemma, which we learned from~\cite{Cnossen2023} (see the proof of Corollary 3.42 there):
\begin{lem}[Mayer--Vietoris]
  If
  \begin{equation}
    \label{eq:pushout-of-anima}
    \begin{tikzcd}
      K\arrow[r,"f"]\arrow[d,"g"]
      & L\arrow[d,"h"]\\
      K'\arrow[r,"f'"]
      & L'
    \end{tikzcd}
  \end{equation}
  is a pushout square in the category of anima, then the induced square
  \[
    \begin{tikzcd}
      (hf)_\sharp\SS_K\arrow[r]\arrow[d]
      & h_\sharp\SS_L\arrow[d]\\
      f_\sharp'\SS_{K'}\arrow[r]
      & \SS_{L'}
    \end{tikzcd}
  \]
  is also a pushout in $\Sp^{L'}$
\end{lem}
\begin{proof}
  By the Yoneda lemma, we must show that
  \[
    \begin{tikzcd}
      \Map (\SS_{L'},E)\arrow[r]
      \arrow[d] & \Map (f_\sharp'\SS_{K'},E)\arrow[d]\\
      \Map (h_\sharp\SS_L,E)\arrow[r]
      & \Map ((hf)_\sharp \SS_K,E)
    \end{tikzcd}
  \]
  is a pullback in the category of anima for every $E\in\Sp^{L'}$. But adjoining over identifies this diagram with
  \begin{equation}
    \label{eq:pullback-of-mapping-spaces}
    \begin{tikzcd}
      \Map (\SS_{L'},E)\arrow[r]
      \arrow[d] & \Map ((f')^*\SS_{L'},(f')^*E)\arrow[d]\\
      \Map (h^*\SS_{L'},h^*E)\arrow[r]
      & \Map ((hf)^* \SS_{L'},(hf)^*E).
    \end{tikzcd}
  \end{equation}
  Note that since~\eqref{eq:pushout-of-anima} is a pushout it induces a pullback on parametrized spectra. The fact that~\eqref{eq:pullback-of-mapping-spaces} is a pullback then follows from the usual description of mapping spaces in pullbacks of categories.
\end{proof}
\begin{proof}[Proof of~\cref{thm:torsion-is-torsion}]
  We may assume that $X$ and $Y$ are both connected, and that the $0$-skeleta of both complexes are singletons. Let $x$ denote the unique $0$-cell of $X$. Let $i_k\colon X^k\hookrightarrow X$ and $i_k'\colon X^k\hookrightarrow X^{k+1}$ denote the inclusion maps.

  The skeletal filtration of $X$ gives rise to a $\lbrack 0,\dim X\rbrack$-indexed filtration
  \begin{equation}
    \label{eq:X-skeletal-filtration}
    \mathscr F\colon (i_0)_?\WallObject{X^0}\to
    (i_1)_?\WallObject{X^1}\to\cdots
    \to \WallObject X,
  \end{equation}
  of $ \WallObject X$. It will be convenient to refine this filtration further by picking a total ordering of the $k$-cells for each $k$, which we combine into a total ordering of all the cells of $X$
  \[
    e_1 < e_2 < \cdots < e_n
  \]
  by imposing $e < e'$ if $\dim e < \dim (e')$. We put $X_l = e_1\cup e_2\cup\cdots\cup e_l$, and let $\iota_l\colon X_l\hookrightarrow X$ and $\iota_l'\colon X_l\hookrightarrow X_{l+1}$ denote the inclusions. We then have a resulting filtration
  \begin{equation}
    \label{eq:X-skeletal-filtration}
    \mathscr G\colon (\iota_0)_?\WallObject{X_0}\to
    (i_1)_?\WallObject{X_1}\to\cdots
    \to \WallObject X,
  \end{equation}
For each graded piece
  \[
    \grpiece^l_{\mathscr G}\WallObject{X}
    \simeq \cofiber ((\iota_{l-1})_?\WallObject{X_{l-1}}\to (\iota_{l})_?\WallObject{X_{l}})
  \]
  of the resulting filtration, \cref{lem:induced-map-on-parametrized-spectra} supplies an identification
  \[
    \grpiece^l_{\mathscr G}\WallObject{X}
    \simeq (\iota_{l})_\sharp \cofiber ((\iota_{l-1}')_\sharp\SS_{X_{l-1}}\to\SS_{X_l}).
  \]
  We also have an associated cell attachment pushout
  \[
    \begin{tikzcd}
      S^k\arrow[r,"a"]\arrow[d,"p"]
      & \type{X_{l-1}}\arrow[d,"\iota_{l-1}'"]\\
      \pt\arrow[r,"x"]
      & \type{X_l}
    \end{tikzcd}
  \]
  in the category of anima, and thus get a Mayer--Vietoris pushout square
  \[
    \begin{tikzcd}
      (xp)_\sharp\SS_{S^k} \arrow[r]\arrow[d]
      & (\iota_{l-1}')_\sharp\SS_{\type{X_{l-1}}} \arrow[d]\\
      x_\sharp\SS\arrow[r] & \SS_{\type{X_{l}}}
    \end{tikzcd}
  \]
  Thus we get an equivalence between $\cofiber ((\iota_{l-1}')_\sharp\SS_{\type{X_{l-1}}}\to \SS_{\type{X_{l}}} )$ with the cofiber of $x_\sharp (p_\sharp\SS_{S^k}\to\SS )$, and $p_\sharp\SS_{S^k}\to\SS$ is simply the map $\SS\lbrack S^k\rbrack\to\SS$ induced by the projection $S^k\to\pt$. The latter is canonically equivalent to $\susp^{k+1}\SS$, so we get an equivalence
  \[
    \cofiber ((\iota_{l-1}')_\sharp\SS_{\type{X_{l-1}}}\to \SS_{\type{X_{l}}} )\simeq x_\sharp\susp^{k+1}\SS\simeq \susp^{k+1}\SS\lbrack\desusp\type{X_{l}}\rbrack.
  \]
 Applying the base change morphism $(\iota_{k})_\sharp = \SS\lbrack\desusp \type X\rbrack\otimes_{\SS\lbrack\desusp \type{X_{l}}\rbrack}({-})$, we get
  \[
    \BNass (\grpiece_{\mathscr G}^k\WallObject X)\simeq
    \SS\lbrack\desusp \type X\rbrack\otimes_{\SS\lbrack\desusp \type{X_{l}}\rbrack}\susp^{k+1}\SS\lbrack\desusp\type{X_{l}}\rbrack
    \simeq\susp^{k+1}\SS\lbrack\desusp{\type X}\rbrack.
  \]
  Summarizing what we have done so far, we have a path from $\WallObject X$ to $\sum_{e\in\mathrm{cells}(X)}F_e$ in $\desusp^\infty\K (\ccoShv X)$, where $F_e = \grpiece_{\mathscr G}^l\WallObject X$ for the unique index $l$ such that $e = e_l$, and we have seen that $\BNass (F_e)\simeq\susp^{\dim e}\SS\lbrack\desusp{\type X}\rbrack\in\Sp^{\type X}$ for each $e$. The same argument gives a path from $f_?\WallObject X$ to $\sum_{e\in\mathrm{cells}(Y)}G_e$, where $\BNass (G_e) = \susp^{\dim e}\SS\lbrack\desusp{\type X}\rbrack$. We define $E$ and $E'$ as in~\cref{subsec: recollection on whitehead}, and pick a bijection $\phi$ between $E$ and $E'$. As in that section, we find a path from
  \[
    \ZZ\lbrack\pi_1X\rbrack\otimes_{\SS\lbrack\desusp\type X\rbrack}(\BNass (F_e) + \desusp \BNass (G_e))
  \]
  to $0\in\desusp^\infty\K (\ZZ\lbrack\pi_1X\rbrack )$, and hence also from $\BNass (F_e) + \desusp \BNass (G_e)$ to $0$ in $\desusp^\infty\K (\Sp^{\type X} )$. But then by~\cref{fact:monomorphism}, this lifts to a path from $F_e + \desusp G_e$ to $0\in\desusp^\infty\K (\ccoShv X)$. Summing all of these together, we get a path from $\Tors f$ to $0\in\desusp^\infty\K (\ccoShv X)$ that lifts the one described in~\cref{subsec: recollection on whitehead}. But then the connecting homomorphism sends the $0$-based loop $\lbrack\tilde\gamma\rbrack\in\uppi_1\K (\Sp^{\type X})$ to the endpoint of its ($0$-based) lift in $\desusp^\infty\K(\ccoShv X )$, namely $\Tors f$.
\end{proof}
\subsection{Basic properties of the torsion cosheaf}
The assignment $f\mapsto\uptau (f)$ satisfies a number of rules that facilitate the calculation of $\uptau (f)$, see for instance the list in~\cite[p.~21, Thm~2.1]{Luck2002}. Aside from homotopy invariance, these rules are also satisfied by the torsion cosheaf $\Tors f$. In fact, it is easier to show these rules for the torsion cosheaf than using the classical construction of Whitehead torsion.
\begin{prop}[Composition formula]
  Let $Z\xrightarrow gY\xrightarrow fX$ be a sequence of maps between compact ANRs. There is a fiber sequence
  \[
    f_?\Tors g\to \Tors{fg}\to\Tors f.
  \]
  In particular, if $f$ and $g$ are both homotopy equivalences, we have $\lbrack \Tors{fg}\rbrack = f_*\lbrack \Tors g\rbrack + \lbrack\Tors f\rbrack\in\K_0 (\WhiteheadCat X)$.
\end{prop}
\begin{proof}
  Since units and counits compose, we find that the map $(fg)_?\WallObject Z\to\WallObject X$ factors as the composition of $f_?(g_?\WallObject Z\to\WallObject Y)$ and $f_?\WallObject Y\to\WallObject X$. The claim now follows from the third isomorphism theorem.
\end{proof}
\begin{rmk}
  Although the cosheaf $\Tors f$ is not itself homotopy invariant, the composition formula lets us explain the homotopy invariance of the Whitehead torsion $\uptau (f) = \lbrack\Tors f\rbrack$ in a way that is native to the Whitehead category. Namely, let $h\colon Y\times I\to X$ be a homotopy with $h\vert_{Y\times 0} = f$ and $h\vert_{Y\times 1} = g$. The composition formula gives fiber sequences
  \[
    h_?\Tors{i_0}\to\Tors{f}\to\Tors{h}
    \quad\text{and}\quad
    h_?\Tors{i_1}\to\Tors{g}\to\Tors{h},
  \]
  where $i_k\colon Y\times k\hookrightarrow Y\times I$ denotes the inclusion for $k=0,1$. The homotopy invariance of $\uptau (f)$ is therefore equivalent to the statement that $\Tors{i_0}$ vanishes in K-theory,\footnote{The statement that $\Tors{i_1}$ follows from this by flipping the interval.} which via Whitehead's theorem is the statement that $i_0$ is a simple homotopy equivalence.
\end{rmk}
\begin{prop}[Sum formula]\label{prop:sum-formula}
  Let $f\colon Y\to X$ be a map between compact ANRs, and let $X = X_0\cap X_1$ be a closed cover of $X$. Suppose that $X_0$, $X_1$, $X_{01} \coloneq X_0\cap X_1$, $Y_0\coloneq f^{-1}(X_0)$, $Y_1\coloneq f^{-1}(X_1)$, and $Y_{01}\coloneq f^{-1}(X_{01})$ are all ANRs.\footnote{This holds for instance if they are neighborhood retracts in their ambient ANRs.}

  Let $f_0\colon Y_0\to X_0$, $f_1\colon Y_1\to X_1$, and $f_{01}\colon Y_{01}\to X_{01}$ denote the restrictions of $f$, and let $i_0\colon X_0\hookrightarrow X$, $i_1\colon X_1\hookrightarrow X$, and $i_{01}\colon X_{01}\hookrightarrow X$ denote the inclusions.
  There is a pushout square
  \[
    \begin{tikzcd}
      (i_{01})_?\Tors{f_{01}}\arrow[r]\arrow[d]
      & (i_1)_?\Tors{f_1}\arrow[d] \\
      (i_0)_?\Tors{f_0}\arrow[r]
      & \Tors f.
    \end{tikzcd}
  \]
  In particular,if $f_0$, $f_1$, and $f_{01}$ are all homotopy equivalences, then  we have $\lbrack\Tors f\rbrack = (i_0)_*\lbrack\Tors{f_0}\rbrack + (i_1)_*\lbrack\Tors{f_1}\rbrack - (i_{01})_*\lbrack\Tors{f_{01}}\rbrack\in\K_0(\WhiteheadCat X)$.
\end{prop}
\begin{proof}
  The torsion cosheaf $\Tors f$ is the cofiber of the map of cosheaves $Y_\sharp f^*\to X_\sharp$. Passing to right adjoints, this corresponds to the map $X^*\to f_*Y^*$. Descent for the closed cover gives a pushout diagram of endofunctors $\Shv X\to\Shv X$ of the form
  \[
    \begin{tikzcd}
      \id\arrow[r]\arrow[d]
      & (i^{\vphantom{*}}_1)_*i_1^*\arrow[d]\\
      (i^{\vphantom{*}}_1)_*i_1^*\arrow[r]
      & (i^{\vphantom{*}}_{01})_*i_{01}^*
    \end{tikzcd}
  \]
  We thus get a cubical commutative diagram
  \begin{equation}
    \label{eq:sum-formula-cubical}
    \begin{tikzcd}
      X^*\arrow[rr]\arrow[dd]\arrow[dr]
      && (i^{\vphantom{*}}_1)_*i_1^*X^*\arrow[dr]\arrow[dd]\\
      & f_*Y^* \arrow[rr]\arrow[dd] && (i^{\vphantom{*}}_1)_*i_1^*f_*Y^*\arrow[dd]\\
      (i^{\vphantom{*}}_1)_*i_1^*X^*\arrow[rr]\arrow[dr]
      & & (i^{\vphantom{*}}_{01})_*i_{01}^*X^*\arrow[dr]\\
      & f_*Y^* (i^{\vphantom{*}}_0)_*i_0^*\arrow[rr] && (i^{\vphantom{*}}_{01})_*i_{01}^*f_*Y^*
    \end{tikzcd}
  \end{equation}
  in which the back and front faces are pushouts. Proper base change identifies the front face with the diagram
  \[
    \begin{tikzcd}
      f_*Y^*\arrow[r]\arrow[d]
      & (f_*(\iota^{\vphantom{*}}_1)_*\iota_1^*Y^*\arrow[d]\\
      f_*(\iota^{\vphantom{*}}_1)_*\iota_1^*Y^*\arrow[r]
      & f_*(\iota^{\vphantom{*}}_{01})_*\iota_{01}^*Y^*,
    \end{tikzcd}
  \]
  where we have used $\iota_k$ to denote the inclusion of $Y_k$ into $Y$. The desired pushout square now comes from passing back to left adjoints in~\eqref{eq:sum-formula-cubical} and taking cofibers in the inwards-outwards direction.
\end{proof}
As in~\cite[p.~40]{Luck2002}, one can also deduce a product formula from the preceding two results. We leave this as an exercise to the reader.
\subsection{The torsion cosheaf of a simplicial map}
It is instructive to consider what it means for a map $f\colon Y\to X$ to have $\Tors f\simeq 0$. Note that by~\cref{thm:torsion-is-torsion} and Whitehead's theorem, the condition that $\Tors f\simeq 0$ is stronger than requiring $f$ to be a simple homotopy equivalence.
\begin{thm}\label{thm:acyclic-map}
  Let $f\colon Y\to X$ be a simplicial map between compact triangulated spaces. The follow are equivalent:
  \begin{enumerate}[label=(\roman*)]
  \item $\Tors f\simeq 0$;
  \item $\redHomology^* (|f^{-1}(\mathrm{Star}_X^\circ (\sigma ))|;\ZZ )\simeq 0$ for every simplex $\sigma\in X$; and
  \item $\redHomology^* (|f^{-1}(\sigma )|;\ZZ )\simeq 0$ for every simplex $\sigma\in X$.
  \end{enumerate}
\end{thm}
We will need the following fact:
\begin{lem}\label{lem:star-retracts-onto-preimage}
  Let $f\colon Y\to X$ be a simplicial map between triangulated spaces. For every simplex $\sigma\in X$, the inclusion
  \[
    |f^{-1}(\sigma )|\hookrightarrow |f^{-1}(\mathrm{Star}_X^\circ (\sigma ))|
  \]
  is a homotopy equivalence.
\end{lem}
\begin{proof}
  We can identify $Y$ with the subspace of $\RR\lbrack Y^0\rbrack$ consisting of linear combinations
  \[
    y = \sum_{v\in Y^0}y_v v
  \]
  such that (i) $\sum_{v\in Y^0}y_v = 1$; (ii) $y_v \geq 0$ for each $v\in Y^0$; and (iii) the set $\lbrace v\in Y^0\mid y_v > 0\rbrace$ forms a simplex in $Y$. Given such $y$, we put \[G(y) = \sum_{v\in f^{-1}(\sigma )}y_v.\] Then $G(y) > 0$ if and only if $y\in |f^{-1}(\mathrm{Star}_X^\circ (\sigma ))|$, and $G(y) = 1$ if and only if $y\in f^{-1}(\sigma )$. For $y\in |f^{-1}(\mathrm{Star}_X^\circ (\sigma ))|$, we can then put
  \[
    y_v^t =
    \begin{cases}
      \frac{y_v}{(1-t) + tG(y)},&\text{if }v\in f^{-1}(\sigma ),\\
      \frac{(1-t)y_v}{(1-t) + tG(y)},&\text{otherwise},
    \end{cases}
  \]
  for $0\leq t\leq 1$. We can then define a deformation retraction of $|f^{-1}(\mathrm{Star}_X^\circ (\sigma ))|$ onto $|f^{-1}(\sigma )|$ by sending $(t,y)$ to $\sum_{v\in Y^0}y_v^tv$.
\end{proof}
\begin{proof}[Proof of~\cref{thm:acyclic-map}]
  The equivalence of (ii) and (iii) is an immediate consequence of~\cref{lem:star-retracts-onto-preimage}. To show that (i) implies (ii), suppose that $\Tors f\simeq 0$. Then we have that \[\SS\lbrack f^{-1}(\mathrm{Star}_X^\circ (\sigma ))\rbrack\to\SS\lbrack \mathrm{Star}_X^\circ (\sigma ))\rbrack\] is an equivalence, and base changing to $\ZZ$ we get the desired statement about integral homology.

  Finally, we will show that (iii) implies (i). By a repeated application of~\cref{prop:sum-formula}, we find that
  \[
    \Tors f\simeq\colim_{\sigma\in X}(i_\sigma )_?\Tors{f_{\sigma}},
  \]
  where $i_\sigma\colon \sigma \hookrightarrow X$ denotes the inclusion and $f_\sigma = f\vert_{f^{-1}(\sigma )}\colon f^{-1}(\sigma )\to \sigma $. It will thus suffice to show that $\Tors{f_{\sigma}}\simeq 0$ for each simplex $\sigma\in X$. But $\sigma$ has a basis consisting of subsets of the form $\mathrm{Star}_{\sigma '}^\circ (v )$, where $\sigma '$ is an iterated barycentric subdivision of $\sigma$ and $v$ is a vertex in $\sigma'$. An argument like the one in the proof of~\cref{lem:star-retracts-onto-preimage} shows that $f^{-1}(\sigma ')$ is a deformation retract of $f^{-1}(\sigma )$, and it then follows from the lemma that $f^{-1}(\mathrm{Star}_{\sigma '}^\circ (v ))$ is contractible. For a general open subset $U\subseteq X$, the claim now follows by covering $U$ with open subsets of the form $\mathrm{Star}_{\sigma '}^\circ (v )$ and noting that $\lbrace \mathrm{Star}_{\sigma '}^\circ (v )\rbrace$ and $\lbrace f^{-1}(\mathrm{Star}_{\sigma '}^\circ (v ))\rbrace$ are good covers of $U$ and $f^{-1}(U)$, whence Borsuk's nerve theorem implies the desired conclusion since $\mathrm{Star}_{\sigma '}^\circ (v )\mapsto f^{-1}(\mathrm{Star}_{\sigma '}^\circ (v ))$ induces an isomorphism on nerves. 
\end{proof}
\begin{cor}
  If $f\colon Y\to X$ is a simplicial homotopy equivalence of compact triangulated spaces such that $f^{-1}(\sigma )$ is acyclic for each simplex $\sigma\in X$, then $f$ is a simple homotopy equivalence.
\end{cor}
\begin{rmk}
  A slightly stronger condition than the one appearing in the corollary is that of $|f^{-1}(\sigma )|$ being contractible for every $\sigma\in\Sigma$. A simplicial map satisfying this condition is called \tdef{contractible} in~\cite{Cohen1967}, and there it is shown that such maps are simple homotopy equivalences. Thus the object $\Tors f$ can be seen as testing whether $f$ is a simple homotopy equivalence in some controlled, standard way; \cref{thm:torsion-is-torsion} says that passing to K-theory precisely relaxes this control.
\end{rmk}

\printbibliography{}
\end{document}